\documentclass[oneside,11pt,reqno]{amsart}
\usepackage{amscd}
\usepackage{amsfonts}
\usepackage{amsxtra}
\usepackage{graphicx}
\usepackage{hyperref}
\usepackage{enumitem}
\usepackage{amsthm}
\usepackage{textcomp}
\usepackage{amsmath}
\usepackage{amssymb}
\usepackage{dsfont}
\usepackage{mathtools}
\usepackage{subcaption}
\usepackage{pinlabel}
\usepackage{floatrow}
\usepackage{soul}
\usepackage{tikz}
\usepackage{color}
\usepackage{float}
\usepackage{eucal}
\usepackage[all]{xy}
\usepackage{multicol}
\usepackage{xcolor}

\definecolor{flaggreen}{RGB}{4,106,56}
\definecolor{flagnavy}{RGB}{6,3,141}
\definecolor{flagsaffron}{RGB}{255,103,31}

\newtheorem{thm}{Theorem}[section]
\newtheorem{prop}[thm]{Proposition}
\newtheorem{lem}[thm]{Lemma}
\newtheorem{cor}[thm]{Corollary}

\theoremstyle{definition}
\newtheorem{defn}[thm]{Definition}
\newtheorem{rem}[thm]{Remark}

\newtheorem{ques}[thm]{Question}

\newtheorem*{prop-MC}{Proposition 4.28}

\newcommand{\norm}[1]{\left\lVert{#1}\right\rVert}

\DeclareMathOperator{\Lk}{Lk}

\DeclareMathOperator{\CAT}{CAT}

\allowdisplaybreaks

\usepackage[font=small]{caption}

\makeatletter
\newtheorem*{rep@theorem}{\rep@title}
\newcommand{\newreptheorem}[2]{%
\newenvironment{rep#1}[1]{%
 \def\rep@title{#2 \ref{##1}}%
 \begin{rep@theorem}}%
 {\end{rep@theorem}}}
\makeatother

\newreptheorem{thm}{Theorem}

\begin{document}

\title[superexponential distortion of free groups in virtually special Groups] {superexponential distortion of free groups in virtually special Groups}

\begin{abstract}
For all integers $k, m > 0$, we construct a virtually special group $G$ containing a finite rank free subgroup $F$ whose distortion function in $G$ grows like $\exp^k(x^m)$. We also construct examples of virtually special groups containing finite rank free subgroups whose distortion functions grow bigger than any iterated exponential. 
\end{abstract} 

\author{Pratit Goswami}
\address{University of Oklahoma, Norman, OK 73019-3103, USA}
\email{Pratit.Goswami-1@ou.edu}

\author{Maya Verma}
\address{University of Oklahoma, Norman, OK 73019-3103, USA}
\email{maya.verma-1@ou.edu}

\maketitle

\section{Introduction}

There has been intense interest in virtually special groups and their subgroups, mainly due to the work of Agol, who solved the virtual fibering conjecture in \cite{Agol}, primarily using machinery from \cite{spcl}. In \cite{spcl}, Haglund and Wise introduced special cube complexes, which are compact nonpositively curved cube complexes whose fundamental groups embed in right-angled Artin groups. A group $G$ is called \emph{virtually special} if it has a finite-index subgroup $H$ such that $H$ is the fundamental group of some special cube complex $X$. The main result in \cite{Agol} is that non-positively curved cube complexes with hyperbolic fundamental groups are virtually special.

We highlight certain properties of virtually special groups that make them an interesting class of groups in geometric group theory. Virtually special groups admit proper cocompact action on CAT($0$) cube complexes, are subgroups of SL$_n(\mathbb{Z})$, are residually finite, and satisfy a stronger version of Tit's alternative [\cite{wiseann}, 14.10]. Although arising from a nice geometric perspective, virtually special groups exhibit subgroups with interesting geometric and algebraic behaviors. One way to understand the geometry of subgroups is to study their \emph{distortion functions} (see Definition~\ref{dist}). One can address the following question: What distortion can finitely generated subgroups of virtually special groups exhibit? In \cite{osapir}, Olshanskii and Sapir proved that the set of distortion functions of finitely generated subgroups of $F_2 \times F_2$ coincides with the set of Dehn functions of finitely presented groups. Now, the Dehn functions of finitely presented groups consist of a huge range of functions by the work of \cite{SBR} and \cite{sapir}. $F_2 \times F_2$ being virtually special, a broad family of functions is known to be distortion functions of finitely generated subgroups of virtually special groups. 

 In this paper, we address the following question: What distortion can finitely presented subgroups of virtually special groups exhibit? We record some known results on this. In \cite{HW}, Hagen and Wise constructed examples of virtually special groups containing finite rank free subgroups whose distortion functions grow like $\exp(x)$. In \cite{BS}, for all integers $m >0$, Brady and Soroko constructed examples of virtually special groups containing finite rank free subgroups whose distortion functions grow like $x^m$. In \cite{DARI}, for all integers $p > q >0$ and $k > 0$, Dani and Riley constructed examples of hyperbolic groups containing finite rank free subgroups whose distortion functions grow like $\exp^k(x^{\frac{p}{q}})$. For $k=1$, these hyperbolic groups turn out to be cubulated by the work from \cite{MR2053602}, and hence virtually special by Theorem~\ref{Agol's theorem}.
Our two main results shed some light on the question above.

For $k \in \mathbb{N}$, we denote the $k$-fold iterated exponential function by $\exp^k(x)$, i.e. $\exp^1(x)=e^x$, and $\exp^k(x)=\exp(\exp^{k-1}(x))$ for $k \geq 2$.

\begin{thm}\label{thm1}
For all integers $m > 0$ and $k > 0$, there exists a virtually special group $G$ containing a finite rank free group $F$ such that the distortion function of $F$ in $G$ grows like $\exp^k(x^m)$.
\end{thm}
 
We give a brief overview of how we prove Theorem ~\ref{thm1}. In Section~\ref{s4}, we first construct $2$-dimensional cubulated hyperbolic groups containing ultra-convex (see Definition~\ref{defn:ultra}) subgroups and exponentially distorted finite rank free subgroups. We refer to them as \emph{building blocks} and our construction parallels that of [\cite{BBD}]. In Section~\ref{s5}, we amalgamate these groups along the ultra-convex subgroup of one with the exponentially distorted subgroup of the other. This kind of amalgamation ensures the group remains CAT($0$), and one expects that the distortion functions compose. To introduce polynomial distortion, we make use of some virtually special groups [\cite{BS}, Theorem A] containing polynomially distorted finite rank free subgroups. These virtually special groups in \cite{BS} are cubulated and we move in a direction to amalgamate them with our groups in Section~\ref{s5}, so that the final group remains cubulated. However, to make these cube complex structures compatible with our desired amalgamation, we introduce \emph{new building blocks} ($2$-vertex complexes) in Section~\ref{s6} [see Remark~\ref{why}]. Finally, we prove Theorem~\ref{thm1} in Section~\ref{s7}. 
 
\begin{thm}\label{thm2}
There exists a virtually special group $G$ containing a finite rank free subgroup $F$ such that the distortion function of $F$ in $G$ is bigger than any iterated exponential.
\end{thm}

In Section~\ref{s8}, we prove Theorem~\ref{thm2}. To accomplish this, we do an HNN extension on our initial building blocks in Section~\ref{s4} such that the stable letter sends the exponentially distorted subgroup to a subgroup of the ultra-convex subgroup. This HNN extension follows the construction of [\cite{BBD}, Theorem 4.1].

\section{Acknowledgements}
The authors would like to thank Noel Brady for his guidance and many helpful conversations on this project. We also thank Max Forester for his comments on an earlier draft of the article. The first author acknowledges support from the 2025 Summer Research Fellowship at the University of Oklahoma. The second author acknowledges support from NSF grant DMS-2405264 during Spring 2025.

\section{Background and Definitions}


In the following section, we present the necessary prerequisites for the upcoming sections. We begin with some definitions associated with cube complexes, all of which can be found in \cite{bridh}.


A \textit{cube complex} is a space constructed by gluing together Euclidean unit cubes $[0,1]^n$ of various dimensions via isometry along their faces. The \textit{link} of a vertex $v$ in a cube complex $X$, denoted by $\Lk(v, X)$, is a simplicial complex whose vertices are edges in $X$ incident to $v$. A set of vertices in $\Lk(v, X)$ spans a simplex if the corresponding edges lie on a face of a cube in $X$ containing $v$. By Gromov's link condition \cite{bridh}, a cube complex $X$ is \textit{nonpositively curved} if and only if the link of each vertex is a CAT$(1)$ simplicial complex. In particular, a  $2$-dimensional cube complex $X$ is nonpositively curved if and only if $\Lk(v, X)$ contains no loop of length less than $2\pi$ for every vertex $v$; this is referred to as the \textit{large link condition}.  
\\We refer the reader to \cite{spcl}  for the definition of special cube complex. Note that the original definition of \emph{special cube complexes} in \cite{spcl} does not require them to be compact. However, in this article, we will only be dealing with compact special cube complexes.

A nonpositively curved cube complex is called \textit{$\CAT(0)$ cube complex } if it is simply connected. We call a group $G$ to be \emph{cubulated} if it acts properly cocompactly on a CAT($0$)
 cube complex.


We now state a version of Gromov's Flat Plane Theorem, which we will use to show certain groups are hyperbolic.

\begin{thm}[Flat Plane Theorem, \cite{bridh}, Theorem 3.1]\label{flt} If a group $\Gamma$ acts properly and cocompactly by isometries on a $\CAT(0)$ space $X$, then $\Gamma$ is word hyperbolic if and only if $X$ does not contain an isometrically embedded copy of the Euclidean plane.
    
\end{thm} 

Finally, the following theorem from \cite{Agol} provides a sufficient condition for a cubulated group to be virtually special.

\begin{thm}\cite[Theroem 1.1]{Agol}\label{Agol's theorem}
    Let $G$ be a word-hyperbolic group acting properly and cocompactly on a $\CAT(0)$ cube complex $X$. Then $G$ has a finite index subgroup $F$ acting specially on $X$. 
\end{thm}

We now define the \emph{distortion function} of $H$ in $G$ for two finitely generated groups $H \leq G$.
\begin{defn}\label{dist}(Subgroup distortion)
For a pair of finitely generated groups, $H \leq G$ with word metric $d_H$ and $d_G$, the distortion of $H$ in $G$ is a function $\delta_H^G: \mathbb{N}  \longrightarrow \mathbb{N}$ given by, 
    $$\delta_H^G(n)= \max \{ d_H(1, w):  w \in H, d_G(1, w) \leq n\}.$$
\end{defn}
The distortion function is well-defined up to Lipschitz equivalence ($\simeq$), that is, it is independent of the choice of word metric.

We record some standard facts about Lipschitz equivalence of functions in the following lemma.

\begin{lem}\label{disteq}

(1) For $m ,n > 0$, $x^m \simeq x^n$ iff $m=n$.

(2) For any $c > 1$, $\exp(x) \simeq c^x$.

(3) For $m, n > 0$, $\exp(x^m) \simeq \exp(x^n)$ iff $m=n$.
\end{lem}

\section{Building Blocks}\label{s4}

In this section, we define a class of groups called the building blocks and show that they are virtually special. Moreover, we demonstrate that they contain finitely generated free subgroups with exponential distortion.

\begin{defn}\label{block1} ($2$-dimensional cubulated hyperbolic building block $P_n$)
For each positive integer $n$, we define a building block group $$P_n \;=\; \langle a_1, \ldots, a_{9n}, t_1, \ldots, t_n \;|\; t_ia_jt_i^{-1} = W_{ij}, 1 \leq i \leq n, 1 \leq j \leq 9n \rangle$$ where $\{W_{ij}\}$ is a collection of positive words of length $9$ in $a_1, \ldots, a_{9n}$ such that each two-letter word $a_ka_l$ appears at most once among the $W_{ij}$'s. To satisfy the above property, we choose $W_{ij}$ to be the consecutive subwords of the following word, defined by Dani Wise in \cite{MR1423338}.
\end{defn}

\begin{defn}\label{wiselong} 
(Wise's long word with no two-letter repetitions) 
For a set $\{a_1, a_2, \cdots , a_m\}$, define the positive word of length $m^2$,
$$\Sigma(a_1, \ldots, a_m)= (a_1a_1a_2a_1a_3 \cdots a_1 a_m)(a_2 a_2 a_3 a_2 a_4 \cdots a_2 a_m) \cdots \cdots (a_{m-1} a_{m-1} a_m) a_m.
$$
Note that, each two-letter subword $a_k a_l$  appears at most once in $\Sigma(a_1, \ldots, a_m)$.
\end{defn}

Next, we define the notion of ultra-convex subcomplexes. This notion is crucial throughout our work for showing when the amalgamation of certain non-positively curved cube complexes remains non-positively curved.

\begin{defn} (Ultra-convex)\label{defn:ultra}
\cite{brady2022superexponentialdehnfunctionsinside}
Let $X$ be a nonpositively curved, piecewise Euclidean $2$-complex, and let $Y$ be a $1$-dimensional subcomplex in $X$.
 Then $Y$ is said to be \textit{ultra-convex} in $X$ if, for every vertex $v \in Y$, any two points in  $\Lk(v, Y)$ are at least $2\pi$  distance apart in $\Lk(v, X)$. This condition is equivalent to the statement that the free group $\pi_1(Y)$ injects into $\pi_1(X)$.
\end{defn}


Here is the main result of this section. In the following proposition, we provide a cubical CAT($0$) hyperbolic structure on the universal cover of the presentation complex $X_n$ of $P_n$. Furthermore, we show that the subgroup $\langle a_1, \cdots a_{9n} \rangle$ is an exponentially distorted free subgroup and the subgroup $\langle t_1, \cdots t_{n} \rangle$ is an ultra-convex convex free subgroup in  $P_n$.

\begin{prop}\label{prop:block1}
Let $n \geq 1$ be an integer. The presentation $2$-complex $X_n$ for $P_n$ can be given a non-positively curved piecewise Euclidean cell complex structure with the following properties:

(1) The $a_j$'s generate a free subgroup $F(a)$ of rank $9n$ whose distortion in $P_n$ is exponential.

(2) The $t_i$'s generate a free subgroup $F(t)$ which is the fundamental group of a rose $R_n$ in $X_n^{(1)}$ such that $R_n$ is ultra-convex in 
$X_n$. 

(3) The universal cover $\widetilde{X_n}$ is hyperbolic.

\end{prop}

\begin{proof}

On each $2$-cell in the presentation $2$-complex $X_n$ of $P_n$, we give a piecewise Euclidean structure by expressing it as a concatenation of five right-angled Euclidean squares, as shown in Figure\ref{5_squres}. 
We verify the large link condition on $X_n$ to conclude that $P_n$ is a CAT($0$) group. The fact that $W_{ij}$'s are positive words in $a_j$'s with no two-letter repetitions ensures the large link condition on $X_n$ as shown in \cite{MR1423338}. Although in Theorem 1.3, \cite{MR1423338} demonstrates the large link condition for a concatenation of five right-angled hyperbolic pentagons, the arguments only use the no repetition of two-letter property of $W_{ij}$'s. Hence, the same arguments follow here.

\begin{figure}[h]
    \centering
    \includegraphics[width=0.5\linewidth]{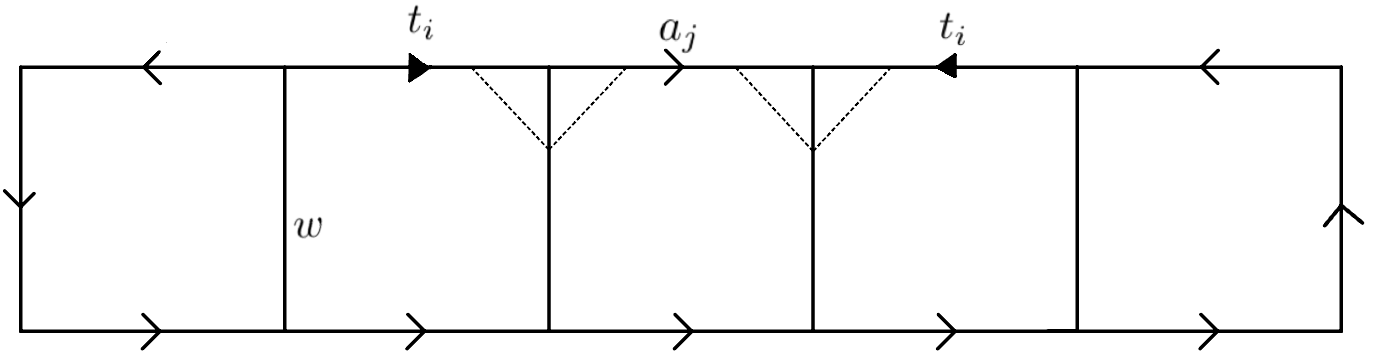}
    \caption{Cell decomposition of $X_n$ into Euclidean squres.}
    \label{5_squres}
\end{figure}

The arguments presented in Proposition 2.2 of \cite{BBD} can be applied here directly to establish (1) and (2). We now prove that $\widetilde{X_n}$ contains no subspace isometric to $\mathds{E}^2$. By the Flat Plane Theorem~\ref{flt}, it follows that the group $P_n$ is hyperbolic.

Denote the unique $0$-cell of $X_n$ by 
$v$. Suppose $\widetilde{X_n}$ contains a subspace isometric to $\mathds{E}^2$. Then any $2$-cell in this subspace must be a lift of a right-angled Euclidean square in $X_n$, as constructed in Figure \ref{5_squres}. This would imply that $\Lk(v,X_n)$ contains a simple loop of length $2\pi$, containing a path with endpoints $a_i^+$ and  $a_j^-$ for some $i$ and $j$, as in Figure~\ref{links}($a$). Due to the way the $2$-cells are cubulated using right-angled Euclidean squares, such a path can be completed to a simple loop of length $2\pi$ in only three possible ways, shown in cases ($b$), ($c$), and ($d$) in Figure~\ref{links}. The no two-letter repetition property of $W_{ij}$ rules out case $(b)$. Furthermore, two vertices of the form $a_i^+$ and $a_l^+$ with $l \neq i$ (or of the form  $a_i^-$ and $a_l^-$) cannot be connected by a single edge in $\Lk(v, X_n)$, since all words in $W_{ij}$ are positive. This eliminates cases ($c$) and ($d$) as well. Therefore, no such loop of length $2\pi$ exists in the link, proving statement (3). It follows from the Flat Plane Theorem that $P_n$ is a hyperbolic group.

\end{proof}

\begin{figure}[h]

\centering
\captionsetup[sub]{font=footnotesize}%
  \subcaptionbox{ \label{fig3:a}}{\includegraphics[width=1.3in]{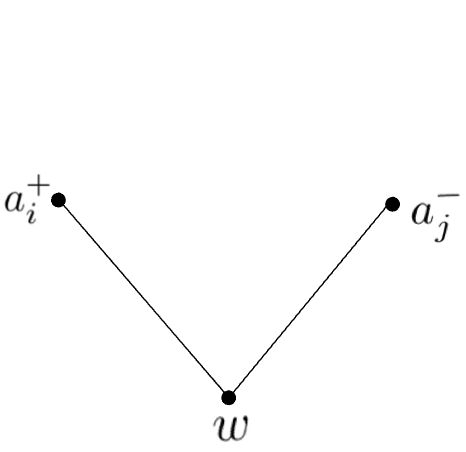}}\hspace{1em}%
  \subcaptionbox{\label{fig3:b}}{\includegraphics[width=1.3in]{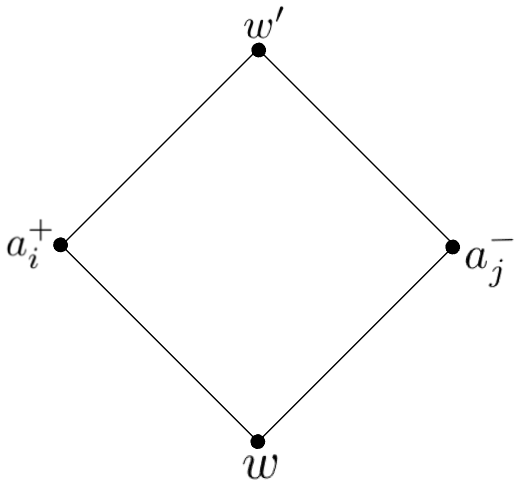}}\hspace{.5em}%
  \subcaptionbox{\label{fig3:c}}{\includegraphics[width=1.3in]{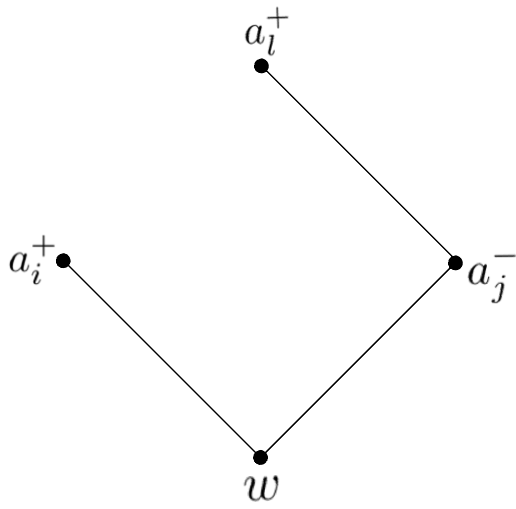}} \hspace{.5em}%
  \subcaptionbox{\label{fig3:d}}{\includegraphics[width=1.3in]{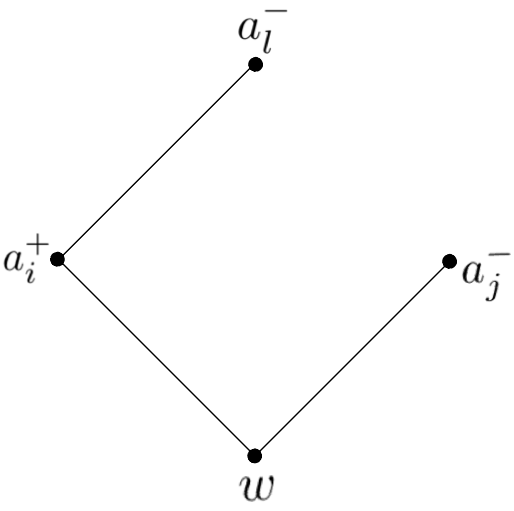}}
  \vspace{.10em}%
\subcaptionbox{ \label{fig3:a'}}{\includegraphics[width=1.3in]{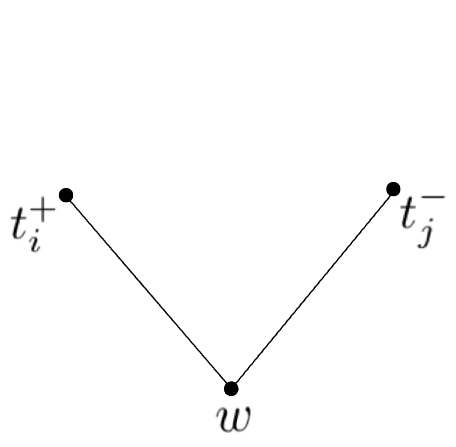}}\hspace{1em}%
\subcaptionbox{ \label{fig3:b'}}{\includegraphics[width=1.3in]{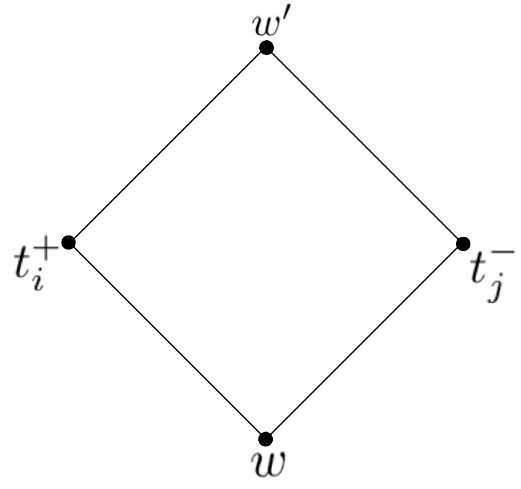}}\hspace{1em}%
\subcaptionbox{ \label{fig3:c'}}{\includegraphics[width=1.3in]{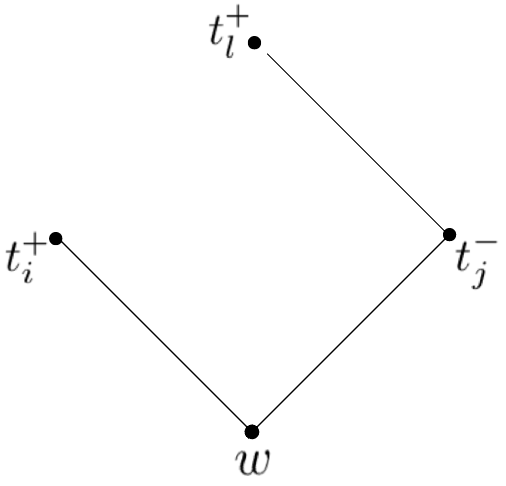}}\hspace{1em}%
\subcaptionbox{ \label{fig3:d'}}{\includegraphics[width=1.3in]{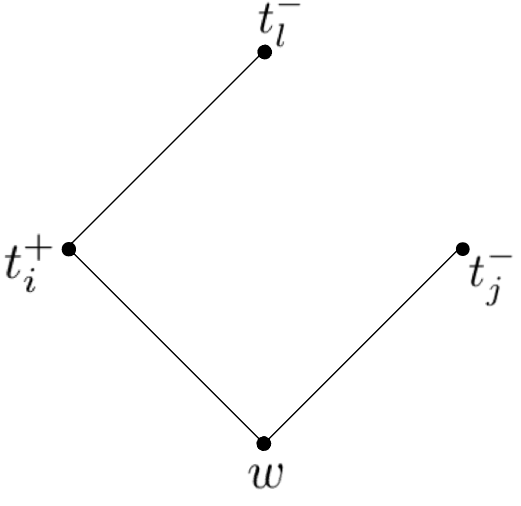}}\hspace{1em}%
   \caption{Paths in $\Lk(v, X_n)$ or $\Lk(v, Z)$. Here, $w$ and $w'$ denote arbitrary unlabeled edges  in $X_n$ and $Z$.}%
  \label{links}%
\end{figure}

\begin{cor}
$P_n$ is a virtually special group containing a free group of rank $9n$ with exponential distortion.
\end{cor}

\begin{proof}
In Proposition~\ref{prop:block1}, we showed that $P_n$ is a cubulated hyperbolic group containing the free group $F(c)$ of rank $9n$ with exponential distortion. By Theorem \ref{Agol's theorem}, $P_n$ is virtually special.
\end{proof}

\section{Iterated Exponential Distortion}\label{s5}

In this section we construct virtually special groups containing finite rank free groups with arbitrary iterated exponential distortion functions. To do so, we amalgamate the groups $P_n$ defined in Section~\ref{s4} by identifying the exponentially distorted subgroup of one with the ultra-convex subgroup of the other. We first prove the groups are cubulated and hyperbolic and conclude virtual specialness by Theorem~\ref{Agol's theorem}. We begin with two results which will be useful for proving hyperbolicity.\\

Using results of Bowditch, Dahmani, and Osin \cite[Theorem 15.4]{DARI} , Dani and Riley proved a combination theorem for hyperbolicity of amalgams $\Gamma = A\ast_CB$. In this article, we will be using their theorem to prove that certain amalgamated groups are hyperbolic. In the interests of self-containment, we record the proof of their theorem here.

\begin{defn}
    A subgroup $H\leq G$ is called \textit{malnormal} if for all $g \in G \backslash H$,
    $H \cap gHg^{-1} = \{e\}$.
\end{defn}

\begin{thm}\label{amalhyp}\cite[Theorem 15.4]{DARI} (Hyperbolicity of amalgams) 
If a finitely generated group $C$ is a subgroup of two hyperbolic groups $A$ and $B$, and $C$ is quasi-convex and malnormal in $A$, then $$\Gamma = A \ast_C B$$ is hyperbolic.

\end{thm}

\begin{proof}
The group $C$ is a finitely generated, quasi-convex, and malnormal subgroup of the hyperbolic group $A$. From [\cite{MR2922380}, Theorem 7.11], we get $A$ is a relatively hyperbolic group with respect to $C$. Using  Theorem 0.1(2) from \cite{MR2026551}, $\Gamma$ is relatively hyperbolic group with respect to $B$. Finally, since a relatively hyperbolic with respect to a hyperbolic group is itself hyperbolic \cite[Corollary 2.41]{MR2182268}, we get that $\Gamma$ is a hyperbolic group. 

\end{proof}

We will be using the following lemma repeatedly to prove that certain subgroups in our constructed hyperbolic groups are malnormal. 

\begin{lem}\label{retract}
Let G be a torsion-free hyperbolic group. Then every retract of $G$ is malnormal.
\end{lem}
\begin{proof}
    Suppose $H$ is a retract of $G$, and let $\phi: G \longrightarrow H$ be a group homomorphism such that $\phi(h)=h$ for all $h \in H$. We claim that if $g^n \in H$, then $g \in H$. Since $g^n \in H$, we have $\phi(g^n)=g^n$. Moreover, $\phi(g)$ commutes with $\phi(g^n)=g^n$. Since $G$ is a torsion-free hyperbolic group, this is possible only if $\phi(g)$ is also a power of $g$. Using again the relation $\phi(g)^n=g^n$ and the fact that $G$ is torsion-free, we conclude that $\phi(g)=g$. Therefore, $g \in H$.
    \\
    Next, suppose $H \cap gHg^{-1} \neq \{e\}$. Then $ghg^{-1}=h_1$ for some nontrivial $h, h_1 \in H$. Applying $\phi$, we obtain  $$\phi(g)h\phi(g^{-1})=\phi(g)\phi(h)\phi(g^{-1})=\phi(h_1)=h_1=ghg^{-1}.$$ Thus, $$g^{-1}\phi(g)h = h g^{-1} \phi(g)$$ and hence $g^{-1}\phi(g)$ lies in the centralizer of $h$. Again, since $G$ is a torsion-free hyperbolic group, this means that $g^{-1} \phi(g)$ is a power of $h$, and therefore an element of $H$. Applying $\phi$ again, we see that
    $$g^{-1} \phi(g)= \phi(g^{-1}) \phi^2(g)= \phi(g^{-1}) \phi(g),$$  which implies that $\phi(g) =g$. Therefore, $g \in H$.
   
\end{proof} 

The following proposition is the main result of this section.

\begin{prop}\label{itrexp}
For any integer $k > 0$, there exists a $2$-dimensional cubulated hyperbolic group $H_k$ containing a free subgroup $F$ of rank $9^k$, such that the distorsion function $\delta_F^{H_k}$ satisfies $\delta_F^{H_k} \simeq f^k(x)$, where $f(x)=9^x$.
\end{prop}

\begin{proof}
Using the building blocks of Proposition~\ref{prop:block1}, for any integer $k>0$, we define the group $H_k$ as
$$H_k = P_1 \ast_{F_9} P_9 \ast_{F_{9^2}} P_{9^2} \ast \cdots \cdots \ast_{F_{9^{k-1}}} P_{9^{k-1}}$$\\
where $F_{9^i}$ for $1\leq i \leq k-1$, is a free group of rank $9^i$ which is identified with the exponentially distorted free subgroup of $P_{9^{i-1}}$ and the ultra-convex free subgroup of $P_{9^i}$.

The group homomorphism $\phi : P_n \rightarrow F(t)$, defined by $\phi(t_r)=t_r$ and $\phi(a_j)=1$, is a retraction of $P_n$ onto $F(t)$. Consequently, $F_{9^i}$ is a retract of $P_{9^i}$ for all $1 \leq i \leq k-1$,  implying that $F_{9^i}$ is a quasi-convex subgroup of $P_{9^i}$. Furthermore, since $P_n$ is torsion free group  and hyperbolic  (by Proposition~\ref{prop:block1}), it follows from Lemma~\ref{retract} that $F_{9^i}$ is a malnormal subgroup of $P_{9^i}$.  Repeatedly applying Theorem~\ref{amalhyp}, we conclude by induction that the group $H_k$ is hyperbolic.

Let $Y_k$ denote the presentation $2$-complex of $H_k$. Then $ Y_1 = X_1$ is locally CAT$(0)$  by Proposition \ref{prop:block1}.
From the inclusion $Y_1 \subset Y_2 \subset \cdots \subset Y_k$, the large link condition can be checked by induction on $k$.  The proof is the same as Theorem 1.3 in 
\cite{BBD}.

Additionally, similar to Theorem 1.3 in 
\cite{BBD}, it can be shown that for the exponentially distorted subgroup $F_{9^{k}}$ of $P_{9^{k-1}}$, $\delta^{H_k}_{F_{9^k}} = f^k(x)$, where $f(x)=9^x$.


\end{proof}

For a fixed integer $m > 0$, we define $H_{k,m} = P_{4m} \ast_{F_{9(4m)}} P_{9(4m)} \ast_{F_{9^2(4m)}} \ast \ldots \ldots \ast_{F_{9^{k-1}(4m)}} P_{9^{k-1}(4m)}$. Following the arguments in the proof above, one observes that $H_{k,m}$ is also a $2$-dimensional cubulated hyperbolic group containing a free subgroup with iterated exponential distortion. This leads us to the following corollary.

\begin{cor}\label{hkm}
For any integer $k,m > 0$, $H_{k,m}$ is a $2$-dimensional cubulated hyperbolic group containing a free subgroup $F$ of rank $9^k(4m)$, such that $\delta_F^{H_{k,m}} \simeq f^k(x)$, where $f(x)=9^x$.
\end{cor}



\section{Building Block With A 2-Vertex Complex}\label{s6}

\begin{figure}[h]
\centering
\captionsetup[sub]{font=footnotesize}%
  \subcaptionbox{$1$-skeleton of $Z_m.$\label{fig3:Z_m1}}{\includegraphics[width=2in]{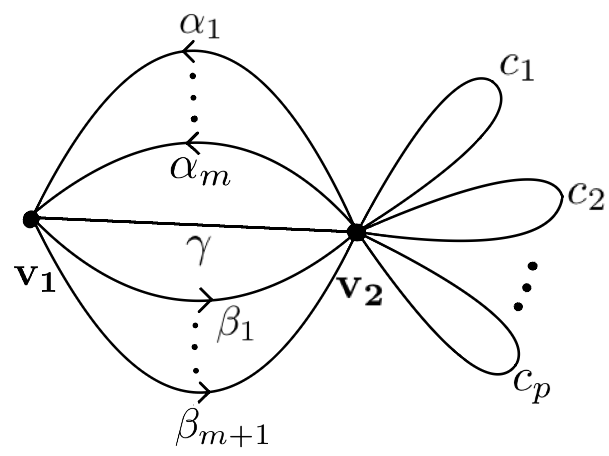}}\hspace{3em}%
  \subcaptionbox{$1$-skeleton of $Z_m'.$\label{fig3:Z_m'}}{\includegraphics[width=2in]{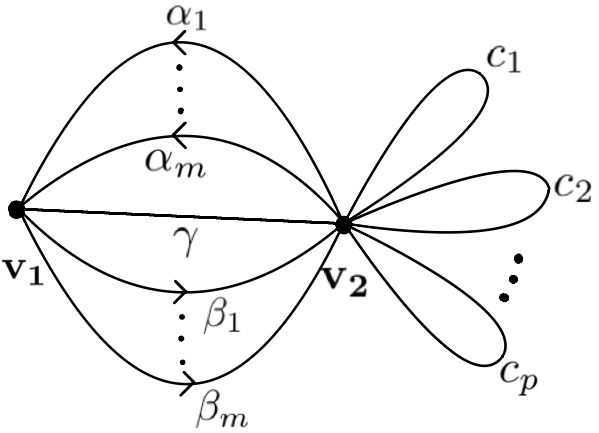}}
  \vspace{.10em}%

\subcaptionbox{The subcomplex $S_m$ of $Z_m.$ \label{fig3:S_m1}}{\includegraphics[width=1.6in]{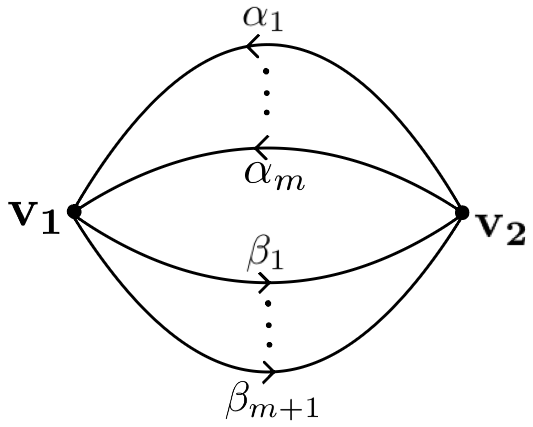}}\hspace{8em}%
\subcaptionbox{The subcomplex $S_m'$ of $Z_m'.$ \label{fig3:S_m'1}}{\includegraphics[width=1.6in]{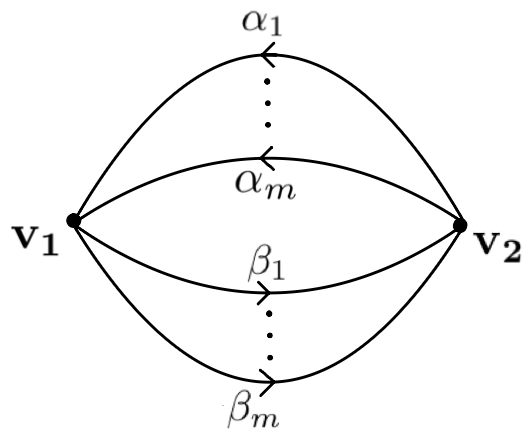}}\hspace{1em}%
   \caption{The $1$-skeletons of $Z_m$ and $Z_m'$ and their respective subcomplexes $S_m$ and $S_m'.$}%
  \label{Zns and Sns}%
\end{figure}

\begin{figure}[h]
    \centering
    \includegraphics[width=0.5\linewidth]{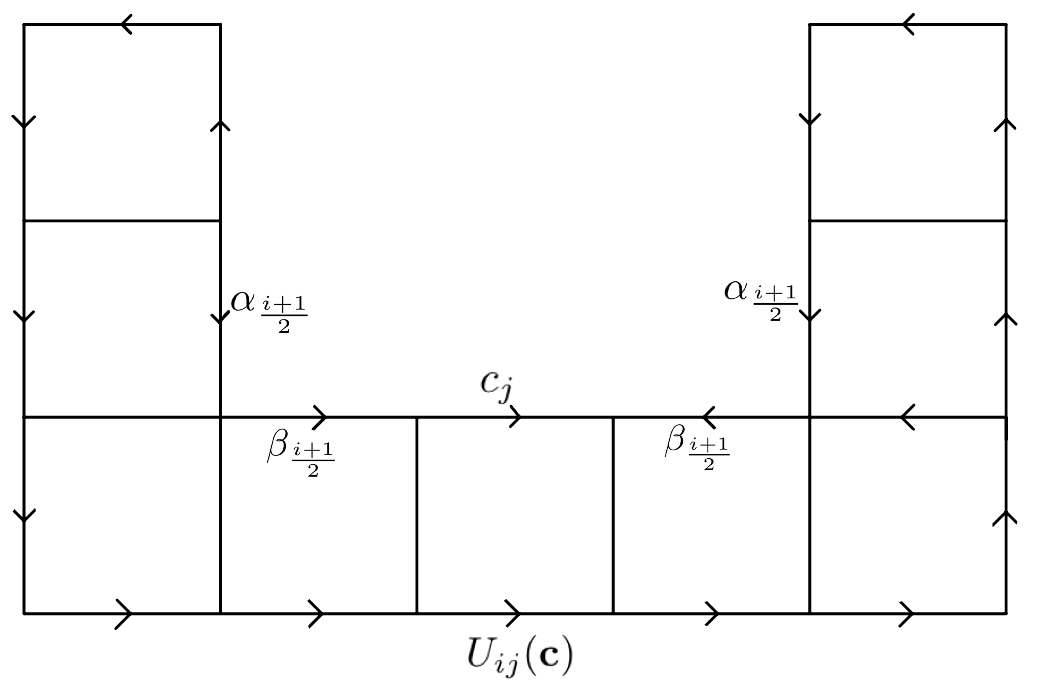}
    \caption{Cell decomposition of $Z_m$ and $Z_m'$ into Euclidean squares. This is for the first relation as in Definition~\ref{block2}. Similar cell decompositions exist for the other three types of relations.}
    \label{cube cx structure on Zns}
\end{figure}

We construct $2$-dimensional cube complexes with $2$-vertices such that their fundamental groups are hyperbolic and CAT($0$). From an algebraic standpoint, we prove that the groups contain exponentially distorted free subgroups and ultra-convex free subgroups. The central idea behind these construction is similar to that of $P_n$ in Section~\ref{s4}.  

\begin{defn}\label{block2} ($2$-dimensional cubulated hyperbolic building block $Q_m$ and $Q'_m$)
For each positive even integer $m$, we define building block groups $Q_m$ and $Q'_m$ to be the fundamental group of the $2$-complexes $Z_m$ and $Z'_m$, respectively, based at the vertex $v_2$. The $1$-skeletons of these complexes are shown in Figure \ref{Zns and Sns}. The $2$-cells are attached along the boundaries as per the following attaching maps 

\begin{align*}
& \alpha_{\frac{i+1}{2}} \beta_{\frac{i+1}{2}} c_j \beta^{-1}_{\frac{i+1}{2}} \alpha^{-1}_{\frac{i+1}{2}}U_{ij}^{-1}, \; i \; \text{odd}, \; 1 \leq j \leq p\\
& \alpha_{\frac{m+k+1}{2}} \beta_{\frac{k+1}{2}} c_j \beta^{-1}_{\frac{k+1}{2}} \alpha^{-1}_{\frac{m+k+1}{2}} U_{kj}^{-1}, \; k \; \text{odd}, \; 1 \leq j \leq p\\
& \beta^{-1}_{\frac{i+2}{2}} \alpha^{-1}_{\frac{i}{2}} c_j \alpha_{\frac{i}{2}} \beta_{\frac{i+2}{2}} U_{ij}^{-1}, \; i \; \text{even}, \; 1 \leq j \leq p\\
& \beta^{-1}_{\frac{m+k+2}{2}} \alpha^{-1}_{\frac{k}{2}} U_{ij}^{-1} c_j \alpha_{\frac{k}{2}} \beta_{\frac{m+k+2}{2}} U_{kj}^{-1}, \; k \; \text{even}, \; 1 \leq j \leq p\\
\end{align*}
to construct $Z_m$ ($1 \leq i, k \leq m$) and $Z'_m$ ($1 \leq i \leq m, 1 \leq k \leq m-1$), where each $\{U_{ij}\}$ and $\{U_{kj}\}$ is a positive words of length $15$ in $c_1, \cdots c_{p}$ satisfying the no two-letter repetition property as defined in ~\ref{wiselong}. The complexes $Z_m$ and $Z_m'$ are given a piecewise Euclidean structure by decomposing each $2$-cell into a concatenation of right-angled Euclidean squares, as illustrated in Figure \ref{cube cx structure on Zns}.  We denote the specific 1-dimensional subcomplex of $Z_m$ (or $Z'_m$) by $S_m$ (or $S'_m$) as shown in Figure \ref{Zns and Sns}.



Following the constraints outlined in the Definition ~\ref{wiselong}, a CAT$(0)$ structure on $Q_m$ requires that $p^2 \geq 15(2m)p$.
Similarly, to endow $Q_m'$ with a CAT$(0)$ structure, we require $p^2 \geq 15(2m-1)p$. For the remainder of the article, we work with $p=36m$.


\end{defn}

\begin{prop}\label{prop:block2}
The $2$-complex $Z_m$ can be given a non-positively curved cube complex structure with the following properties:

(1) The $c_j$'s generate a free subgroup $F(c)$ of rank $36m$ whose distortion in $Q_m$ is exponential.


(2) The universal cover $\widetilde{Z_m}$ is hyperbolic.

\end{prop}

\begin{proof}

Each $2$-cell of $Z_m$ has a piecewise Euclidean structure since it is defined by expressing it as a concatenation of nine right-angled Euclidean squares, as shown in Figure \ref{cube cx structure on Zns}. Similar to the proof of Proposition~\ref{prop:block1}, 
one can verify the large link condition on $Z_m$, which implies that $Q_m$ is a CAT($0$) group.

We now show that the $c_j$'s generate a free subgroup $F(c)$ of rank $36m$. 


The group $Q_m$ admits a graph of group structure with underlying graph a bouquet of $2m$ circles. The vertex and edge groups are the free group $F(c)$, generated by $c_1, \ldots, c_{36m}$, with associated monomorphisms defined as follows on the $i$th edge:

$$id_{F(c)} : F(c) \rightarrow F(c), \hspace{10pt} id_{F(c)}(c_j)=c_j,$$
$$\phi_i : F(c) \rightarrow F(c), \hspace{10pt}  \phi_i(c_j)=U_{ij},$$
for $1 \leq j \leq 36m$. Each word $U_{ij}$ is a positive word in the generators $c_1, \ldots, c_{36m}$, so the map $\phi_i$ can be viewed as a homomorphism from $F(c)$ to $F(c)$.



We observe that $F(c)$ is a free group of rank $36m$. Using Stallings' folding algorithm, one can verify that the map $\phi_i$ is an injective (see the details in Proposition 2.2 of \cite{BBD}). This implies that $F(c)$ embeds in $Q_m$ as a subgroup. Applying the same argument as in \cite[Proposition 2.2]{BBD}, we conclude that the distortion of $F(c)$ in $Q_m$ satisfies
$$\delta_{F(c)}^{Q_m} (x) \simeq 14^x.$$ This completes the proof of (1).




Finally, the arguments presented in Proposition~\ref{prop:block1} show that $\widetilde{Z_m}$ contains no subspace isometric to $\mathds{E}^2$. By Flat Plane Theorem ~\ref{flt}, this shows that the group $Q_m$ is hyperbolic.

\end{proof}

The following proposition records analogous properties for $Q'_m$ and $Z'_m$. The proof follows the same outline as that of  Proposition~\ref{prop:block2}.

\begin{prop}\label{prop:block2'}
The $2$-complex $Z'_m$ can be given a non-positively curved cube complex structure with the following properties:

(1) The $c_j$'s generate a free subgroup $F(c)$ of rank $36m$ whose distortion in $Q'_m$ is exponential.


(2) The universal cover $\widetilde{Z'_m}$ is hyperbolic.

\end{prop}

\section{${\exp}^k(x^m)$ Distortion}\label{s7}

In this section, we prove Theorem~\ref{thm1}. One aspect of Theorem~\ref{thm1} is to get $\exp^k(x^m)$ distortion functions. So far, we have only constructed groups containing exponential or iterated exponential distortion functions. To introduce polynomial distortion, we revisit the construction in \cite{BS} of certain virtually special groups containing polynomially distorted free subgroups. The central idea for proving Theorem~\ref{thm1} is again to amalgamate groups by identifying the distorted subgroup of one with the ultra-convex subgroup of the other. However, the use of $2$-vertex complexes in our construction in Theorem~\ref{thm1} introduces some challenges in showing certain complexes are non-positively curved cube complexes. To overcome these we take help of Theorem~\ref{algtopo} and Lemma~\ref{ultra convexity chaining}.\\

For each $m \in \mathbb{N}$ and $k=0, \ldots, m$, Brady and Soroko, in \cite{BS}, constructed a family of virtually special cubulated groups $G_{m,k}$ with the following presentation: $$G_{m,k} = \langle \; s_1, \ldots, s_{m+k+1} \; | \;[s_i,s_{i+1}]=1 \textit{ for } i = 1, \ldots, m ; \; s^{-1}_{m+j+1}s_js_{m+j+1}=s_{m+j} \textit{ for } j=1, \ldots, k \; \rangle.$$ 

Let $K_{m,k}$ denote the presentation $2$-complex associated with the above presentation of $G_{m,k}$. 

\begin{prop}\cite[Proposition 5.1]{BS} \label{kmk} The presentation complex $K_{m,k}$ defined above is a non-positively curved cube complex.
\end{prop}

The group $G_{m,k}$ also admits free-by-cyclic structure \cite[Proposition 5.4]{BS}; that is:

$$G_{m,k} \cong F_{m+k} \rtimes_{\phi_{m,k}} \mathbb{Z},$$
where $F_{m+k} = \langle A_1, \ldots, A_m,B_1, \ldots, B_k \rangle$, $\mathbb{Z}= \langle t \rangle$ with $A_i = s^{-1}_{i+1}s_i$ for $1 \leq i \leq m$, and $B_j = s^{-1}_{m+j+1}s_j$ for $ 1 \leq j \leq k$. The automorphism $\phi_{m,k}$ is defined as follows (bar denotes the inverse of an element in the group):

\begin{align*}
\phi_{m,k} : \; & A_1 \mapsto A_1\\
& A_2 \mapsto A_1 (A_2) {\overline{A}_1}\\
& A_3 \mapsto A_1 A_2 (A_3) {\overline{A}_2} {\overline{A}_1}\\
& \vdots\\
& A_m \mapsto A_1 A_2 \ldots A_{m-1} (A_m) \overline{A}_{m-1} \ldots \overline{A}_1\\
& B_1 \mapsto A_1 A_2 \ldots A_m (B_1)\\
& B_2 \mapsto A_1 A_2 \ldots A_m (B_1 B_2) \overline{A}_1\\
& B_3 \mapsto A_1 A_2 \ldots A_m (B_1 B_2 B_3) \overline{A}_2 \overline{A}_1\\
& \vdots\\
& B_k \mapsto A_1 A_2 \ldots A_m (B_1 B_2 \ldots B_k) \overline{A}_{k-1} \overline{A}_{k-2} \ldots \overline{A}_2 \overline{A}_1 \\
\end{align*}


\begin{defn} 
    The \textit{growth} of an automorphism $\phi: F \rightarrow F$, where $F$ is a finitely generated free group with basis $\mathcal{A}$, is defined by 
    $$gr_{\phi, \mathcal{A}}(n)=  \max_{a \in \mathcal{A}}  \norm{\phi^n(a)}_\mathcal{A} .$$
\end{defn}

\begin{thm} \cite[Theorem A]{BS} \label{BSvspcl}  For each positive even integer $m$, the groups $G_{m,m} \cong F_{2m} \rtimes_{\phi_{m,m}} \mathbb{Z}$ and  $G_{m,m-1} \cong F_{2m-1} \rtimes_{\phi_{m,m-1}} \mathbb{Z}$ are virtually special, with growth function satisfying $gr_{\phi_{m,m}} \sim n^m$ and $gr_{\phi_{m,m-1}} \sim n^{m-1}$, respectively.
\end{thm}


In the following lemma we compute the distortion functions of $F_{2m}$ and $F_{2m-1}$ in $G_{m,m}$ and $G_{m,m-1}$, respectively.

\begin{lem}\label{BSdist} 
The distortion functions of $F_{2m}$ in $G_{m,m}$ satisfies $\delta_{F_{2m}}^{G_{m,m}} (n) \simeq n^{m+1}$  and the distortion of  ${F_{2m-1}}$ in $G_{m,m-1}$ satisfies $\delta_{F_{2m-1}}^{G_{m,m-1}} (n) \simeq n^{m}$.
\end{lem}

\begin{proof}

We prove that $\delta_{F_{2m}}^{G_{m,m}} (n) \simeq n^{m+1}$ and a similar argument can be used to prove $\delta_{F_{2m-1}}^{G_{m,m-1}} (n) \simeq n^{m}$. From hereon we will denote $\phi_{m,m}$ by $\phi$. For a fixed $n \in \mathbb{N}$, one gets $gr_{\phi} (n) = \norm{\phi^n(B_m)} \simeq n^m$, from the results in \cite[Section 8]{BS}. Now consider the word $w=t^n{B_m}^nt^{-n}$ in $G$.  Clearly, $|w|_G \leq 3n$. Now,

\begin{align*}
w &=t^n B_m \ldots B_m t^{-n}\\
&=t^{n-1}tB_mt^{-1} \ldots tB_mt^{-1}t^{-(n-1)}\\
&=t^{n-1}\phi(B_m) \ldots \phi(B_m)t^{-(n-1)}\\
&=t^{n-2}t\phi(B_m)t^{-1} \ldots t\phi(B_m)t^{-1} t^{-(n-2)}\\
&=t^{n-2}\phi^2(B_m) \ldots \phi^2(B_m)t^{-(n-2)}\\
& \vdots\\
&=\underbrace{
    \phi^n(B_m)  \ldots \phi^n(B_m)
    }_{\text{$n$~elements}}
\end{align*}
We now compute the length of $w$ as a word in $F_{2m}$. To accomplish this, we will be using certain results concerning $\phi$ from \cite{BS}. 

Firstly, by \cite[Lemma 8.4]{BS}, we have $$\phi^n(B_1)=A_1^nA_2^n \ldots A_m^nB_1.$$ 
\\
By the definition of $\phi$, we also have 

$$\phi(B_2)=\phi(B_1)B_2 \overline{A}_1.$$\\
For all $k \geq 2$, from \cite[Claim 8.8]{BS}, one gets $$\phi^n(B_{k+1})=\phi^n(B_k)\phi^{n-1}(A_1 \ldots A_{k-1})\phi^{n-1}(B_{k+1}){\phi^{n-1}(A_1 \ldots A_k)}^{-1}.$$
\\
Taking $k+1=m$ in the above equality, we get $$\phi^n(B_{m})=\phi^n(B_{m-1})\phi^{n-1}(A_1 \ldots A_{m-2})\phi^{n-1}(B_{m}){\phi^{n-1}(A_1 \ldots A_{m-1})}^{-1}.$$
\\
Using the above equality, one observes that when written in terms of the letters $A_i$ and $B_j$, $\phi^n(B_m)$ will begin with $\phi^n(B_1)=A_1^nA_2^n \ldots A_m^nB_1$.

Finally, from [\cite{BS}, Lemma 8.6], we get that $${\phi^{n-1}(A_1 \ldots A_{m-1})} = A_1^{n}A_2^n \ldots A_{m-2}^nA_{m-1} \overline{A}^n_{m-2} \ldots \overline{A}^n_2 \overline{A}^n_1.$$
Thus, $\norm{\phi^n(B_m).\phi^n(B_m)} = 2n^m - 2n(m-2)-1$ (since, the $A_i$ and $B_i$ s generate a free group, no further cancellation is possible). Hence, $$\norm{\phi^n(B_m) \ldots \phi^n(B_m)} = n.n^{m} - (2n(m-2)-1)(n-1) \simeq n^{m+1},$$ since $m \geq 2$.

So, $|w|_{F_{m,m}} = \norm{\phi^n(B_m) \ldots \phi^n(B_m)} \simeq n^{m+1}$. This establishes the lower bound.

To prove the upper bound, take a word $w$ in $F_{m,m}$ such that $|w|_G \leq n$. Cancelling the innermost $t \ldots t^{-1}$ pairs, $w$ can be expressed as a word in the $a_i$'s of length at most $(\frac{n}{2})^m.n \simeq n^{m+1}$. Hence, $\delta_{F_{m,m}}^G (n) \simeq n^{m+1}$.

\end{proof}
The following lemma will be useful for proving Lemma~\ref{linear} and follows from a similar proof of \cite[Lemma 8.6]{BS}.
\begin{lem}\label{productA} For all $n , k > 0$,
${\phi^{n}(A_1 \ldots A_{k})} = A_1^{n}A_2^n \ldots A_{k-1}^nA_{k} \overline{A}^n_{k-1} \ldots \overline{A}^n_2 \overline{A}^n_1.$
\end{lem}

\begin{defn}
A word in $G_{m,k}$ with the free-by-cylic presentation is called a \emph{positive} word if it does not contain any inverse letters ($\overline{A_1}, \ldots, \overline{A}_m,  \overline{B_1}, \ldots, \overline{B_k}, \overline{t}$).
\end{defn}


\begin{lem}\label{linear}
For all $n , k > 0$, $\phi^n(B_k) = A_1^n \ldots A_m^n u_{k,n}B_k\overline{A}^{n}_{k-1} \ldots \overline{A}^{n}_1$, where $u_{k,n}$ is a positive word.
\end{lem}

\begin{proof}
The proof is by double induction on $n$ and $k$. Since $\phi (B_1) = A_1 \ldots A_m B_1$, the base case of the induction holds true. By induction hypothesis, $$\phi(B_k) = A_1 \ldots A_m u_{k,1}B_k \overline{A}_{k-1} \ldots \overline{A}_1.$$  
From the definition of $\phi$,  we have $$\phi(B_{k+1}) = \phi (B_k) . (A_1 \ldots A_{k-1}). B_{k+1} . (\overline{A}_k \ldots \overline{A}_1).$$ 
Using the form of $\phi(B_k)$ in the inductive hypothesis, it clearly follows that $\phi(B_{k+1})$ has the desired form.

Assume for a fixed $k$,
$$\phi^j(B_k) = A_1^j \ldots A_m^j u_{k,j}B_k\overline{A}^{j}_{k-1} \ldots \overline{A}^{j}_1, \;  \forall \; k \geq 1.$$
To end the double induction argument, we need to show $\phi^{j+1} (B_k)$ has the desired form. One observes from the definition of $\phi$ that $$\phi^{j+1} (B_k) = \phi^j (A_1 \ldots A_m) \phi^j (B_1 \ldots B_k) \phi^j (\overline{A}_{k-1} \ldots \overline{A}_1).$$
The form of $\phi^j(B_k)$ in the inductive hypothesis and Lemma~\ref{productA} gives us the desired form for $\phi^{j+1} (B_{k}).$






\end{proof}

We state a gluing theorem for adjunction spaces, a version of which appeared as Theorem 7.5.7 in \cite{brown}.

\begin{thm}\label{algtopo} (Gluing theorem for adjunction spaces)
If $X$, $Y$, and $Z$ are topological spaces, with $A \subset X$, a map $f : A \rightarrow Y$, and a homotopy equivalence $\phi : Y \rightarrow Z$, such that $(X,A)$ has the homotopy extension property, then the space $Y \bigcup_f X$ is homotopy equivalent to $Z \bigcup_{\phi \circ f} X$.
\end{thm}

The following lemma shows that when $2$-dimensional non-positively curved cube complexes are glued along ultra-convex subcomplexes, the resulting complex remains non-positively curved.

\begin{lem}\label{ultra convexity chaining} (Ultra convex gluing)
    Let $X$ and $Y$ be $2$-dimensional non-positively curved complexes, and let $Z \subset X, Y$ be a $1$-dimensional subcomplex. If $Z$ is ultra-convex in $X$ (or in $Y$), then the space $ X\sqcup_ZY$ is also a $2$-dimensional non-positively curved complex.
\end{lem}
\begin{proof}
The space $X\sqcup_ZY$ is a $2$-dimensional complex; so by the large link condition, it is enough to show that the link of every vertex in this space contains no loop of length less than $2\pi$. 
Let $v$ be a vertex in $X\sqcup_ZY$. There are two possibilities: 
\begin{enumerate}

    \item  If $v$ is not in $Z$: then it is a vertex from either from $X$ or $Y$ alone. Therefore,  it's link in $X\sqcup_ZY$ is the same as its link in $X$ or $Y$  and since both $X$ and $Y$ are non-positively curved, this link contains no loop shorter than $2\pi$. 
    \item If $v$ is image of a vertex $v_0 \in Z$: suppose $v_0$ maps to a vertex $v_1$ in $X$ and $v_2$ in $Y$. Then the link of $v$ in $X\sqcup_ZY$ is obtained by gluing the link of $v_1$ in $X$ with the link of $v_2$ in $Y$, along the link of $v_0$ in $Z$.
$$\Lk(v, X\sqcup_ZY) = \Lk(v_1, X) \sqcup_{\Lk(v, Z)} \Lk(v_2, Y)$$ 
As a result, any loop in $\Lk(v, X\sqcup_ZY)$ must either be entirely within $\Lk(v, X)$ or $\Lk(v, Y)$, in which case it is at least $2\pi$ in length, or it must partially be contained in both. But since $Z$ in ultra-convex in $X$, any two points in $\Lk(v_0, Z)$ are at least $2\pi$ separated in $\Lk(v_1, X)$, so any such loop must still have length at least $2\pi$. 
   \end{enumerate}
    
    \end{proof}

We now prove our first main theorem.

\begin{thm}\label{mainkm}
For all integers $m > 0$ and $k > 0$, there exists a virtually special group $G$ containing a finite rank free group $F$ such that the distortion function of $F$ in $G$ grows like $\exp^k(x^m)$.
\end{thm}

\begin{proof}

We first construct $G$ and $F$ when $m=1$. In Proposition~\ref{itrexp}, we have shown that $H_k$ is a cubulated hyperbolic group containing the free group $F_{9^k}$ of rank $9^k$ such that $\delta_{F_{9^k}}^{H_k} \simeq f^k(x)$, where $f(x)=9^x$. By Theorem~\ref{Agol's theorem}, $H_k$ is virtually special. Taking $G = H_k$ and $F = F_{9^k}$ proves this case.

For the remainder of the proof let $m \geq 2$.
For all integers $k > 0$ and each positive even integer $m$, we first construct a virtually special group $G$ containing a free group $F$ of rank $9^{k-1}(4m)$ such that $\delta_H^G \simeq \exp^k(x^{m+1})$.

\begin{figure}[h]
    \centering
    \includegraphics[width=0.4\linewidth]{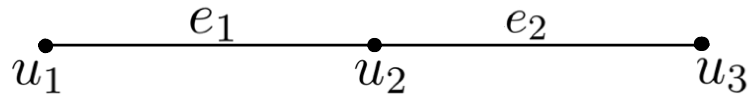}
    \caption{Underlying graph for the graph of groups structure on $G.$}
    \label{graph with 2 vertices}
\end{figure}

Let $G$ be the fundamental group of the graph of groups whose underlying graph is given in Figure \ref{graph with 2 vertices}.
The vertex group at $u_1$,  $u_2$, and  $u_3$ are $G_{m,m}$, $Q_m$, and $H_{k-1,m}$, respectively. The edge group at $e_1$ and $e_2$ are $F_{2m}$ and $F_{36m}$, respectively. Let $Y_{k,m}$, and $K_{m,m}$ be presentation complex of $H_{k,m}$, and $G_{m,m}$, respectively. Also, recall $R_{36m}$, $Z_m$, and $S_m$ defined in Proposition~\ref{prop:block1}, Definition~\ref{block2} respectively. Then, the injective homomorphisms from edges groups to vertex groups are as follows;

\begin{itemize}
    \item The injective homomorphism $\phi_\ast:F_{2m} \rightarrow G_{m,m}$ is induced from the map $\phi : S_m \rightarrow K_{m,m}$ defined as $\phi (\alpha_i)=s^{-1}_{2i}$ for $ 1 \leq i \leq m$, and $\phi (\beta_i)=s_{2i-1},$ for $1 \leq i \leq m+1$,
    \item The injective homomorphism $\iota_\ast:F_{2m} \rightarrow Q_m$ is induced from the natural inclusion map $\iota_\ast: S_m \rightarrow Z_m$,

    \item The injective homomorphism $\psi_\ast:F_{36m} \rightarrow Q_m$ is induced from the map $\psi : R_{36m} \rightarrow Z_{m}$ defined as $\psi (c_i)=c_i, \; 1 \leq i \leq 36m$,
    \item Lastly, the injective homomorphism $\rho_\ast:F_{36m} \rightarrow H_{k-1,m}$ is induced from the map $\rho : R_{36m} \rightarrow Y_{k-1,m}$ defined as $\rho (c_i)=a_i, \; 1 \leq i \leq 36m$.
\end{itemize}

Note that $G=G_{m,m} \ast_{F_{2m}} Q_m \ast_{F_{36m}} H_{k-1,m}$ such that $F_{2m}$ is identified with the polynomially distorted subgroup in $G_{m,m}$ and with the ultra-convex subgroup in $Q_m$, while $F_{36m}$ is identified with the exponentially distorted subgroup in $Q_m$ and with the ultra-convex subgroup $H_{k-1,m}$.

Consider the total space $C_G$ of graph of space for $G$ with vertex space at $u_1$, $u_2$, and $u_3$, to be $K_{m,m}$, $Z_m$ and $Y_{k-1, m}$, and the edge space at $e_1$ and $e_2$ to be $S_m$
and $R_{36m}$, respectively. The maps from edge space to vertex space are $\phi$, $\iota$, $\psi$, and $\rho$. Then by Lemma \ref{algtopo}, $C_G$ is homotopy equivalent to the quotient space
$$C=K_{m,m} \bigsqcup_{S_m} Z_{m} \bigsqcup_{R_{36m}} Y_{k-1,m}.$$



We now prove that $C$ is a non-positively curved cube complex, which will imply that $G$ is a cubulated group. Since the cube complexes $K_{m,m}$, $Z_m$, and $Y_{k-1,m}$ are glued along edges, $C$ also has a cube complex structure.  By Proposition \ref{kmk} and Proposition \ref{prop:block2}, both $K_{m,m}$ and $Z_m$ are $2$-dimensional non-positively curved cube complexes. Furthermore, any two vertices $\alpha^{\pm}$ and $\beta^{\pm}$ in the Lk($v_1, Z_m$) is exactly $\frac{3\pi}{2}$ separated. Hence, any loop in Lk($v, K_{m,m} \bigsqcup_{S_m} Z_{m}$) is atleast of length $2\pi$, where $v$ denotes the unique vertex in $K_{m,m} \bigsqcup_{S_m} Z_{m}$. Thus, by large link condition,  $K_{m,m} \bigsqcup_{S_m} Z_{m}$ is also a non-positively curved cube complex. Now, since $R_{36m}$ is ultraconvex in a $2$-dimensional non-positively curved cube complex $Y_{k-1,m}$, Lemma~\ref{ultra convexity chaining} ensures that $C$ is a non-positively curved cube complex.

We now prove the group $G$ contains a free subgroup $F$ of rank $9^{k-1}(4m)$, whose distortion function in $G$ grows like $\exp^k(x^{m+1})$. By corollary~\ref{hkm}, $H_{k-1,m}$ contains a free subgroup of rank $9^{k-1}(4m)$, say $F$, whose distortion function grows like $\exp^{k-1}(x)$. In $H_{k-1,m}$, let $a^{(i)}_j$, $1 \leq j \leq 9^{i}(4m)$ denote the generators of $F_{9^i(4m)}$, $1 \leq i \leq k-2$. From the results in \cite[Section 8]{BS}, Lemma~\ref{BSdist} and Lemma~\ref{linear}, we get  $t^n \phi^n(B_m) t^{-n} = w_{m,2n}.l_{m,2n}$ where $w_{m,2n}$ is a positive word in the letters $A_i$, $B_j$ and $l_{m,2n}$ is a word in $A^{-1}_i$ s only (i.e does not contain any $A_i$, $B_j$ or $B^{-1}_j$), such that $\norm{l_{m,2n}}_{F_{2m}} \simeq n$. Hence, $\norm{w_{m,2n}}_{F_{2m}} \simeq n^{m+1}$. We define the following sequence of words: 

\begin{align*}
& w_1 = w_{m,2n}\\
& w_2 = w_1 c_1 w^{-1}_1\\
& w_3 = w_2 a^1_1 w^{-1}_2\\
& \vdots\\
& w_{k+1} = w_{k} a^{k-1}_1 w^{-1}_{k}.\\
\end{align*}

One now observes that $w_{k+1} \in F$ with $\norm{w_{k+1}}_G \simeq n$. By Lemma~\ref{disteq}, $\norm{w_{k+1}}_F \simeq \exp^{k}(n^{m+1})$. This establishes the lower bound for the distortion function.

\begin{figure}[h]
    \centering
    \includegraphics[width=0.9\linewidth]{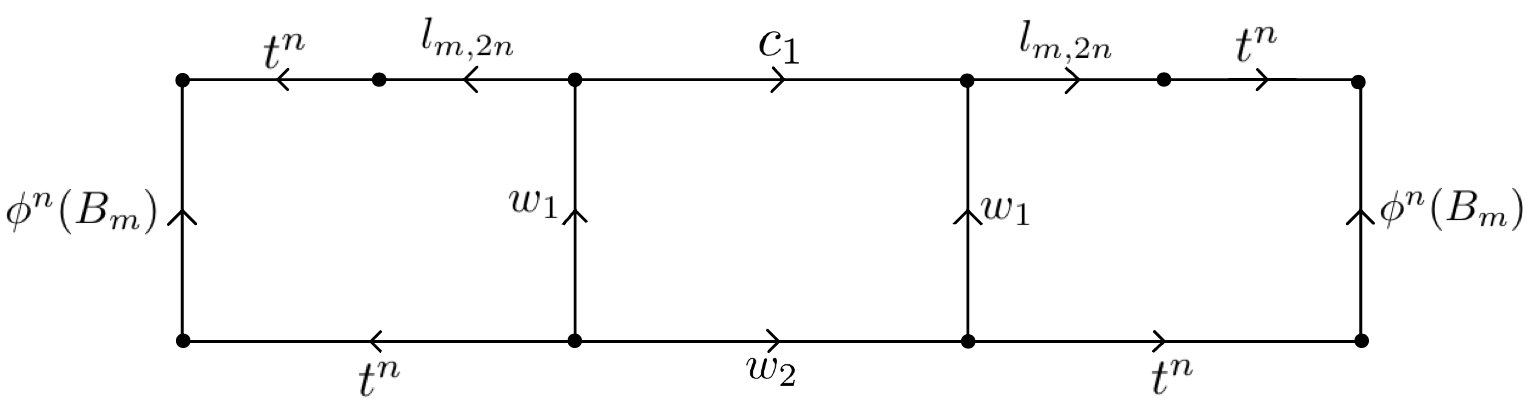}
    \caption{Distortion diagram exhibiting $\exp(x^{m+1})$ distortion for $w_2.$}
\end{figure}

We now prove the upper bound for the distortion. Similar to Proposition~\ref{itrexp}, the distortion of $F$ in $Q_m \ast_{F_{36m}} H_{k-1,m}$ is $\exp^k(x)$. Let $w$ be a word in $F$ such $|w|_G = n$. Cancelling at most $\frac{n}{2}$ innermost $t \ldots t^{-1}$ pairs (here $t$ denotes the stable letter in the free-by-cyclic presentation of $G_{m,k}$), $w$ can be represented by a word in $Q_m \ast_{F_{36m}} H_{k-1,m}$ of length at most $(\frac{n}{2})^{m+1} \simeq n^{m+1}$. Since $\delta_F^{Q_m \ast_{F_{36m}} H_{k-1,m}} \simeq \exp^k(x)$, we conclude that $\delta_F^G$ is bounded above by $\exp^k(x^{m+1})$.

Finally, we show that $G$ is virtually special. $F_{36m}$ is a retract of $H_{k-1,m}$ and hence quasi-convex in $H_{k-1,m}$. By Lemma~\ref{retract}, $F_{36m}$ is malnormal is $H_{k-1,m}$. By Theorem~\ref{amalhyp}, $Q_m \ast_{F_{36m}} H_{k-1,m}$ is hyperbolic. Again, $F_{2m}$ is a retract of $Q_m \ast_{F_{36m}} H_{k-1,m}$ and hence quasi-convex. By Lemma~\ref{retract}, it is also malnormal. So, by \cite[Theorem 2.6]{MR3456181} $Q_m \ast_{F_{36m}} H_{k-1,m}$ is relatively hyperbolic with respect to $F_{2m}$ Hence, by \cite[Theorem 0.1(2)]{MR2026551}, the group $G$ is relatively hyperbolic with respect to the set of conjugates of $G_{m,m}$ in $G$, say $\tilde{C}$. Recall that by Theorem~\ref{BSvspcl}, the groups $G_{m,m}$ are virtually special. Thus $(G, \tilde{C})$ satisfies the conditions of \cite[Hypothesis 1.6]{MR4413505}. Hence, by \cite[Theorem A]{MR4413505}, the group $G$ is virtually special.

To accommodate the remaining cases i.e. $\exp^k(x^{m})$ with $m$ even, we start by defining a graph of spaces for the group $G'=G_{m,m-1} \ast_{F_{2m-1}} Q'_m \ast_{F_{36m}} H_{k-1,m}$. The rest of the proof is identical to the first case. Thus, $G'$ is a virtually special group containing a free group of rank $9^{k-1}(4m)$ with distortion growing like $\exp^k(x^{m})$.

\end{proof}

\begin{rem}\label{why} (Why a $2$-vertex complex?) 
The groups $G_{m,k}$ constructed in \cite{BS} are cubulated and contain a polynomially distorted free subgroup. In contrast, our groups $H_{k,m}$ from Corollary~\ref{hkm} are cubulated groups containing super-exponentially distorted free subgroups. The key distinction between these constructions lies in how the distorted subgroups are embedded into the cubical structure. In $G_{m,m}$, the generators of the distorted free subgroup correspond to diagonals of the squares, rather than to edges of the cubical complex. In contrast, in $H_{k,m}$, the generators of both the distorted and ultra-convex subgroups are actual edges of the $2$-cubes. As a result, amalgamating these two groups along the polynomially distorted subgroup of $G_{m,k}$ and the ultra-convex subgroup of $H_{k,m}$ does not yield a natural cubulation on the resulting group. To address this, we introduce a $2$-vertex complex that serves as an intermediate space in the amalgamation. In this complex, each generator of the ultra-convex subgroup of its fundamental group corresponds to a path consisting of two edges joined at a common vertex. This setup allows us to interpret the diagonals in $G_{m,k}$  as $2$ consecutive edges within $2$-cubes. We then amalgamate this $2$ vertex comoplex along the polynomially distorted subgroup of $G_{m,k}$ and the ultra-convex subgroup of $H_{k,m}$, thereby preserving a consistent cubical structure throughout.
\end{rem}

\section{Higher than any iterated exponential}\label{s8}

We are now ready to prove Theorem~\ref{thm2}. To do so, we perform an HNN extension on $P_n$  (groups defined in Section~\ref{s4}), in which the stable letter maps the exponentially distorted subgroup to the ultra-convex subgroup. The argument follows the same outline as the proof of  \cite[Theorem 4.1]{BBD}.

\begin{thm}\label{bigitr}
There exists a virtually special group $H$ containing a finite rank free subgroup $F$ such that $\delta_F^H$ is bigger than any iterated exponential.
\end{thm}

\begin{proof}
Define $$H = \langle \; a_1,\ldots,a_m, t_1, \ldots , t_n, s \; | \; t_ia_jt_i^{-1} = W_{ij} \; , \; sa_ls^{-1}=W_l \; \rangle,$$ where $W_{ij}$ is a collection of $mn$ positive words of length $9$ in the letters $a_1, \ldots, a_m$ and $W_l$ is a collection of $m$ positive words of length $9$ in the letters $t_1, \ldots, t_n$. Each $W_{ij}$  and $W_k$ is chosen to satisfy the no two-letter repetition property. We chose the $W_{ij}$'s and $W_k$'s to be disjoint subwords of $\Sigma(a_1, \ldots, a_m)$ and $\Sigma(t_1, \ldots t_n)$, respectively, as defined in Definition~\ref{wiselong}. These constraints force the inequalities: $9mn \leq m^2$ and $9m \leq n^2$. As an example, take $n = 9^2$ and $ m = 9^3$.\\

Let $Z$ denote the presentation complex for $H$. The $2$-cells of $Z$ can be given a piecewise Euclidean structure by expressing them as a concatenation of right-angled Euclidean squares, as shown in Figure~\ref{two_5_squares}. An argument analogous to that of Theorem 4.1 in \cite{BBD} then shows that $Z$ is CAT$(0)$. 

\begin{figure}[h]
    \centering
    \includegraphics[width=0.5\linewidth]{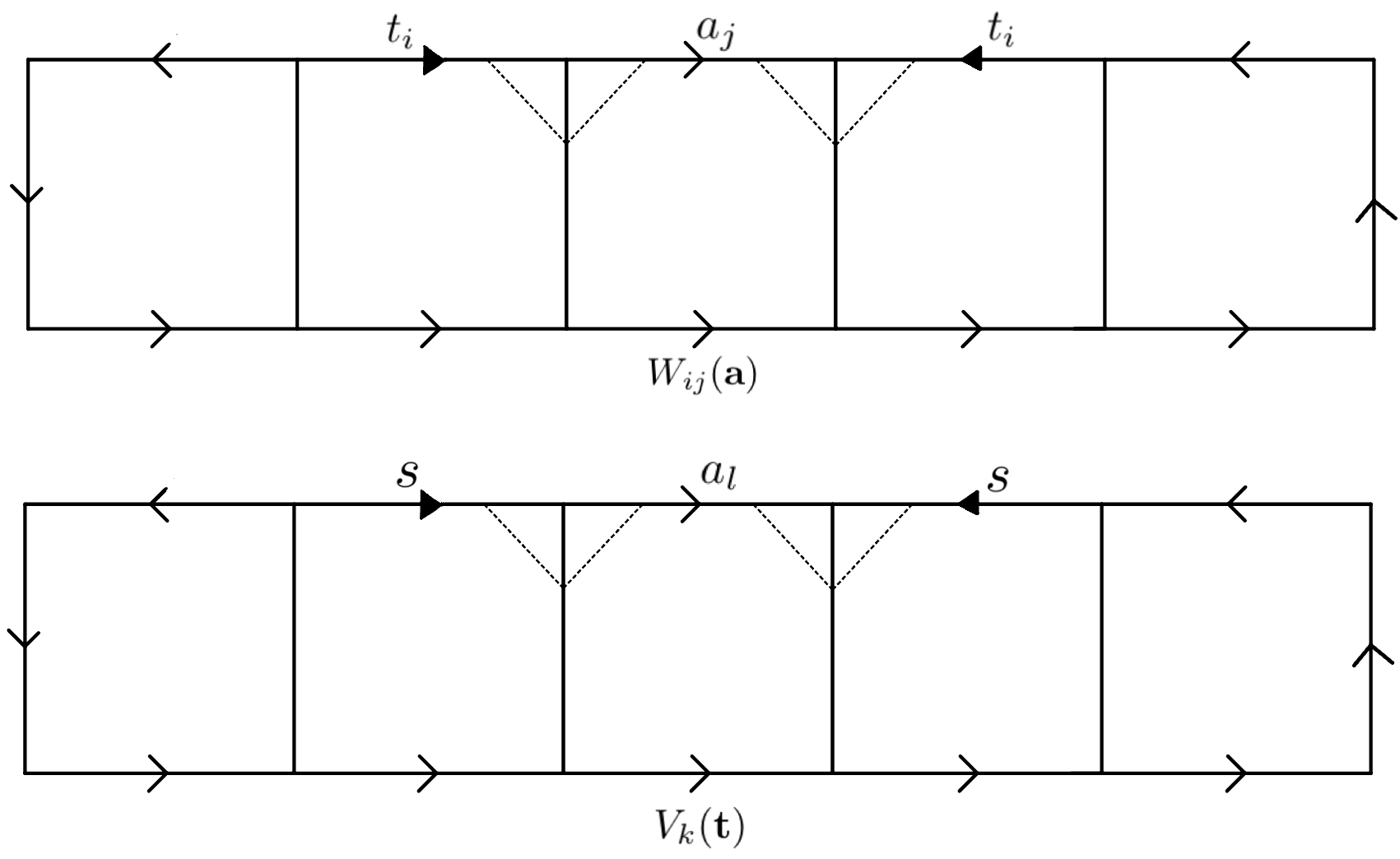}
    \caption{Cell decomposition of $Z$ into Euclidean squres.}
    \label{two_5_squares}
\end{figure}
Using arguments analogous to those in Proposition~\ref{prop:block1} (since in both cases we have the same decomposition of the $2$-relator cells into squares), we can show that $\widetilde{Z}$ contains no subspace isometric to $\mathds{E}^2$. Therefore, by the Flat Plane Theorem, it follows that the group $H$ is hyperbolic.


Lastly, the argument from Theorem 4.1 of \cite{BBD} shows that the free subgroup with the required distortion function is $F= \langle \;a_1, \ldots, a_m \; \rangle \subset H$. Thus, we get a cocompactly cubulated hyperbolic group $H$ containing a free subgroup $F$ whose distortion function $\delta_F^H$ is bigger than any iterated exponential. By Theorem~\ref{Agol's theorem}, the group $H$ is virtually special.

\end{proof}

\begin{rem} (Ackermannian distortion in virtually special groups) We provide some more examples of virtually special groups containing finite rank free groups whose distortion functions are bigger than the any iterated exponential. These examples arise from the work of Brady, Dison, and Riley in \cite{BDR}.

In \cite{DR}, Dison and Riley constructed some CAT($0$), free-by-cyclic, one relator and biautoamtic groups called the \emph{hydra groups}, $G_k$ which contain a free subgroup $H_k$ of rank $k$,  such that the distortion grows like $k$-th Ackermann's functions: $\delta_{H_k}^{G_k} \simeq A_k$. In particular, $A_1(n)=2n,$ $A_2(n)=2^n$, and $A_3(n)$ which is the $n$-fold iterated power of $2$, all occur as distortion functions. In \cite{BS}, Brady and Soroko asked whether the hydra groups $G_k$ are virtually special, which is still an open question.

The group $G_k$ is not hyperbolic, since it contains $\mathbb{Z}^2$ as a subgroup. In \cite{BDR}, Brady, Dison, and Riley constructed a hyperbolic version of the original hydra groups. This construction yields hyperbolic, free-by-cyclic, CAT($0$) groups $\Gamma_k$ containing free subgroups $\Lambda_k$ such that $\delta_{\Lambda_k}^{\Gamma_k} \succeq A_k$. 
For finite-rank free groups, hyperbolic free-by-cyclic groups act freely and cocompactly on CAT($0$) cube complex \cite{HW}. By Theorem Theorem~\ref{Agol's theorem}, cubulated hyperbolic groups are virtually special. Hence we conclude that the groups $\Gamma_k$ are virtually special. This shows that there exists a virtually special group $\Gamma_k$ containing finite rank free group $\Lambda_k$ whose distortion function satisfies 
$\delta_{\Lambda_k}^{\Gamma_k} \succeq A_k$. One observes that $A_3(n)$, which is the $n$-fold iterated power of $2$, grows faster than any iterated exponential. So, other than the examples we constructed in Theorem~\ref{bigitr}, the hyperbolic hydra groups of \cite{BDR} also provides examples of virtually special groups containing free groups whose distortion is bigger than any iterated exponential, namely, $\delta_{\Lambda_3}^{\Gamma_3} \succeq A_3(n)$.
\end{rem}

\section{Questions}

We conclude this paper with some open questions. Let $\mathbb{Z}^+$ denote the set of positive integers.\\
\begin{ques}\label{q1}  Do there exist virtually special groups containing finitely presented subgroups whose distortion functions grow like $x^{\alpha}$, for $\alpha \not\in \mathbb{Z}^{+}$ and $\alpha > 1$?
\end{ques}

\begin{ques}\label{q2} Do there exist virtually special groups containing finitely presented subgroups whose distortion functions grow like $\exp^k(x^{\alpha})$, for $k > 0$ and $\alpha > 1$ ?
\end{ques}

For $k=1$ and a rational number $\alpha > 1$, Theorem 1.1 of \cite{DARI}, and for $\alpha \in \mathbb{Z}^+$, Theorem~\ref{mainkm}, provide positive answer to Question~\ref{q2}.\\

We are not aware of any examples in the literature of CAT($0$) groups containing finitely presented subgroups whose distortion functions grow like $x^{\alpha}$, for $\alpha \not\in \mathbb{Z}^{+}$ and $\alpha > 1$. This prompts the following question.

\begin{ques}\label{q3}  Do there exist CAT($0$) groups containing finitely presented subgroups whose distortion functions grow like $x^{\alpha}$, for $\alpha \not\in \mathbb{Z}_{+}$ and $\alpha > 1$ ?
\end{ques}

In the context of super-exponential distortion of finitely presented subgroups inside CAT($0$) groups, for all integers $k >0$ and $p > q >0$, Dani and Riley in \cite{DARI} constructed hyperbolic groups containing finite rank free groups whose distortion functions grow like $\exp^k(x^{\frac{p}{q}})$. By \cite[Remark 15.6]{DARI}, these hyperbolic groups can be given CAT($0$) (and also CAT($-1$)) structures. 
\bigskip

\bigskip

\bibliographystyle{alpha} 
\bibliography{bibfile}

\end{document}